\documentclass[11pt]{amsart}
\title{An exact degree for multivariate special polynomials}
\author{Rudolph Bronson Perkins}

\usepackage{amsthm}
\usepackage{amsmath}
\usepackage{amssymb}
\usepackage[hidelinks]{hyperref}

\newcommand{\FF}{\mathbb{F}}

\newcommand{\pitilde}{\widetilde{\pi}}

\numberwithin{equation}{section}

\newtheorem{theorem}{\bf Theorem}[section]
\newtheorem{corollary}[theorem]{\bf Corollary}
\newtheorem{lemma}[theorem]{\bf Lemma}
\newtheorem{proposition}[theorem]{\bf Proposition}

\theoremstyle{definition}
\newtheorem{definition}[theorem]{\bf Definition}
\newtheorem{remark}[theorem]{\bf Remark}
\begin{document}

\address{Institut Camille Jordan, Université Claude Bernard Lyon 1, 43 boulevard du 11 novembre 1918, 69622 Villeurbanne cedex, France}
\email{perkins@math.univ-lyon1.fr}

\begin{abstract}
We introduce certain special polynomials in an arbitrary number of indeterminates over a finite field. These polynomials generalize the special polynomials associated to the Goss zeta function and Goss-Dirichlet $L$-functions over the ring of polynomials in one indeterminate over a finite field and also capture the special values at non-positive integers of $L$-series associated to Drinfeld modules over Tate algebras defined over the same ring.

We compute the exact degree in $t_0$ of these special polynomials and show that this degree is an invariant for a natural action of Goss' group of digit permutations. Finally, we characterize the vanishing of these multivariate special polynomials at $t_0=1$. This gives rise to a notion of trivial zeros for our polynomials generalizing that of the Goss zeta function mentioned above.
\end{abstract}

\maketitle

\section{Introduction}
\subsection{Notation}
Let $\FF_q$ be the finite field with $q$ elements of positive characteristic $p$, and let $\theta$ be an indeterminate. Our interest is in the ring $A := \FF_q[\theta]$. We denote by $A_+$ the set of monic polynomials in $A$, and for all $d \geq 0$ we write $A_+(d)$ for those elements of $A_+$ whose degree in $\theta$ equals $d$. We write $\overline\FF_q$ for the algebraic closure of $\FF_q$, and we fix an embedding of $\overline\FF_q$ into the algebraic closure of $\FF_q(\theta)$.

We shall say a $p$-adic integer $\beta$ is \textit{written in base} $q$ whenever we write $\beta = \sum_{i \geq 0} \beta_i q^i$ with $0 \leq \beta_i < q$ for all $i \geq 0$. We define the \textit{length} of a positive integer $\beta$, written in base $q$ as above, to be $l(\beta) := \sum_{i \geq 0}\beta_i$. Of course, note the dependence on $q$ that we omit. Finally, for a rational number $\alpha$, we let $\lfloor \alpha \rfloor \in \mathbb{Z}$ denote the greatest integer less than or equal to $\alpha.$

\subsection{Multivariate Special Polynomials}
The results of this note expand upon the author's work in \cite{Per1} and have been ported over from the author's dissertation \cite{PerDiss}. We study the \textit{multivariate special polynomials}, defined for non-negative integers $\beta_1,\dots,\beta_s$ first by the formal series
\[z(\beta_1,\dots,\beta_s,t_0) := \sum_{d \geq 0}t_0^d\sum_{a \in A_+(d)} \chi_{1}(a)^{\beta_1}\cdots\chi_{s}(a)^{\beta_s} \in \FF_q[t_1,\dots,t_s][[t_0]].\]
Here, for all $a \in A$ and $i = 1,\dots,s$, the symbols $\chi_i(a)$ stand for the images of the maps
\[\chi_i : A \rightarrow \FF_q[t_i] \subseteq \FF_q[t_1,\dots,t_s]\]
determined by $\theta \mapsto t_i$. It follows from our recursive formula, Prop. \ref{sprecursion}, that these power series are in fact polynomials in $t_0$.

\subsection{A Comment on Notation}
Goss points out to us that $s$ is traditionally the coordinate used on his ``complex plane'' $\mathbb{S}_\infty$ (see \cite[Chapter 8]{Gbook} for the definition). In keeping with the established notation so far in the theory of Pellarin's multivariate $L$-series, we will always use $s$ to denote a non-negative integer. As $\mathbb{S}_\infty$ does not appear in this paper, we do not expect any confusion to arise.

\subsection{Drinfeld modules over Tate algebras} 
In \cite{APTR14}, Angl\`es, Pellarin and Tavares-Ribeiro introduce the notion of Drinfeld modules over Tate algebras. The authors associate $L$-series to these new Drinfeld modules, generalizing the multivariate $L$-series studied in \cite{AngPel1}, \cite{GPel}, \cite{Pel13}, and \cite{Per2}. They mention in Example 4.1.3 that the recursive formula of this note (Prop. \ref{sprecursion}) implies that the negative special values of their $L$-series are finite $K$-linear combinations of elements that are algebraic over $\FF_q(t_0,t_1,\dots,t_s)$. Let us briefly describe the connection.

The most basic multivariate special polynomials considered here,
\[z(1,\overset{s}{\dots},1,t_0) := \sum_{d \geq 0} t_0^d \sum_{a \in A_+(d)} \chi_1(a)\cdots\chi_s(a),\]
are an example of the special value at zero of the $L$-series associated to the Drinfeld module over the $s$-variable Tate algebra with parameter (in the notation of \cite{APTR14}) $\alpha = t_0(t_1 - \theta)\cdots(t_s - \theta)$. As we will see in later sections, we may recover all the multivariate special polynomials considered in this note from $z(1,\overset{s}{\dots},1,t_0)$. We also observe that specializing the variables $t_i$ in the algebraic closure of $\FF_q(t_0,t_1,\dots,t_s)$ or in the set $\{\theta^{q^j}\}_{j \geq 0}$ returns the full gauntlet of $L$-series values at negative integers considered in \cite{APTR14}, and such specializations pose no convergence issues as we are simply considering polynomials in this note. One may also observe that the proof of Proposition \ref{sprecursion} below goes through word for word with the $t_i$ viewed in any $\FF_q$-algebra.


\subsection{Goss' Special Polynomials}
The following bit of history will help frame the results of this note. The special polynomials associated to $L$-series in positive characteristic are originally due to Goss, and we direct the reader to \cite[Chapter 8.12]{Gbook} for more information on their arithmetic interpretation. Briefly, one uses these special polynomials to give continuous-analytic continuation, in the sense of Goss, to their associated $L$-series at all finite places of $A$.

For non-negative integers $\beta$, the polynomials $z(\beta,t_0)$ are essentially the special polynomials associated to the Goss zeta function for $A$; the difference between the two is superficial and one re-obtains Goss' special polynomials from $z(\beta, t_0)$ through the change of variables $t_0 \mapsto x^{-1}$ and $t_1 \mapsto \theta$.  It is interesting to observe that the polynomials $z(\beta,t_0)$ may be obtained in many ways through specialization of the variables $t_i$ in the multivariate polynomials $z(\beta_1,\dots,\beta_s,t_0)$, and, perhaps counter-intuitively, such specializations - from several variables to one variable - will be a major tool in deducing properties about the multivariate special polynomials to follow.

The first basic result for $z(\beta,t_0)$ is due to Goss and follows easily from a recursive formula that we generalize in the next section. Of course, the reader should compare the following result with the vanishing of the Riemann zeta function at the negative even integers.

\begin{theorem}[Goss] \label{Gossthm}
The power series $z(\beta,t_0) \in \FF_q[t_1][[t_0]]$ is in fact a polynomial in $\FF_q[t_0,t_1]$. The polynomial vanishes at $t_0=1$ if and only if $\beta>0$ is such that $\beta \equiv 0 \mod (q-1)$ and in this case the zero is simple.
\end{theorem}

\begin{remark} 
The zeros of $z(\beta,t_0)$ at $t_0 = 1$, characterized above, are referred to as the \textit{trivial zeros} of the Goss zeta function. In the past, such trivial zeros have been viewed as arising from Euler factors at infinite primes of classical (characteristic zero valued) $L$-series through Goss' \textit{double-congruence}; we direct the reader to \cite[8.13]{Gbook} for full details, and to \cite{GPel} and \cite{Per1} where such zeros were first noticed and explored. In \cite{Gbook} (see Example 8.13.9 and the preceding paragraph), Goss asks for a method of handling the trivial zeros of $L$-series in positive characteristic without resorting to classical $L$-series. A partial answer to Goss' question is given by the recent results of Angl\`es and Pellarin, and we describe this in Section \ref{APTZsec}. As we will show, one may now see these zeros (but not their simplicity) as coming from ``gamma factors'' in positive characteristic. This new gamma factor is none other than Anderson and Thakur's function $\omega$, and we direct the reader to \cite{PelGamma} for comparisons between $\omega$ and the Euler $\Gamma$-function.
\end{remark}

The next result is stated in \cite[Theorem 1.2 (a)]{B2013}, where B\"ockle attributes it to J. Sheats. It will be a basic tool in our computation (see Theorem \ref{thm:exdegsp} below) of the degree in $t_0$ of $z(\beta_1,\dots,\beta_s,t_0)$.

\begin{theorem}[B\"ockle, Sheats]
For all positive integers $\beta$, the exact degree in $t_0$ of $z(\beta,t_0)$ equals $\min_{i \geq 0}\lfloor l(p^i \beta)/(q-1)\rfloor$.
\end{theorem}

Obtaining the upper-bound in the theorem above is actually quite easy, and this will follow from our recursive formula proved in the next section. Proving the lower bound causes more difficulty, and one employs Sheats' results, as in \cite{B2013}, to do this.

\subsection{Goss' Group of Digit Permutations}
In \cite{Gdigit}, D. Goss introduced, for each $q$ as above, a natural group $S_{(q)}$ of \textit{digit permutations} which act as homeomorphisms of the $p$-adic integers $\mathbb{Z}_p$ (considered with their usual topology). To define the group $S_{(q)}$, one begins by writing each element $\beta \in \mathbb{Z}_p$ uniquely in base $q$ as $\beta = \sum_{i \geq 0}\beta_i q^i$. Then to each permutation $\rho$ of the non-negative integers, one associates a permutation $\rho_*$ of $\mathbb{Z}_p$ through the definition
\[\rho_*\left(\sum_{i \geq 0}\beta_i q^i\right) := \sum_{i \geq 0} \beta_i q^{\rho(i)} \text{ for all } \beta \in \mathbb{Z}_p.\]

\begin{definition}
Goss' group $S_{(q)}$ of digit permutations is the collection of $\mathbb{Z}_p$-homeomorphisms $\rho_*$ arising as $\rho$ ranges over all permutations of the non-negative integers.
\end{definition}

We direct the reader to \cite[Section 6.2]{Gdigit} for the interesting history of how Goss first observed that the degree in $t_0$ of $z(\beta,t_0)$ is an invariant of the natural action of $S_{(q)}$ given by $\rho_*(z(\beta,t_0)) := z(\rho_*\beta,t_0)$. Given the B\"ockle-Sheats' result above, this invariance of degree is now quite easy to see. We will use our exact degree in Theorem \ref{thm:exdegsp} to extend Goss' observation on invariance of degrees to the special polynomials in several variables introduced above (see Theorem \ref{invgen}).

\subsection{On the Overlap with Angl\`es-Pellarin}
In \cite[Section 3]{AngPel1}, similar results are considered as a prerequisite for the proof of rigid-analyticity in the variables $t_1,\dots,t_s$ of Pellarin's series in $s$ indeterminates. Outside of the method of specialization to Goss zeta values and $L$-values, which is a major theme in the life of Pellarin's series anyway, our methods of proof are different. Further, our results concern the special polynomials associated to Goss' extension of Pellarin's series and are slightly finer in general than those of \cite{AngPel1}. Also note that \cite[Lemma 31]{AngPel1} is our Theorem \ref{thmvanchar} minus simplicity of the zero at $t_0 = 1$.

\section{Recursive Nature}
In this section we focus on the special polynomials
\[z(1,\overset{s}{\dots},1,t_0) = \sum_{d \geq 0} t_0^d \sum_{a \in A_+(d)} \chi_{1}(a)\cdots\chi_{s}(a).\] The following recursive formula will be our major tool.

\begin{proposition}\label{sprecursion}
Let $s$ be a non-negative integer. The following recursive formula holds: \newline $z(1,\overset{0}{\dots},1,t_0) = 1$, and for all $s \geq 1$,
\begin{equation}
z(1,\overset{s}{\dots},1,t_0) = 1 - t_0\sum_{\substack{(i_1,\dots,i_s) \in \{0,1\}^s \\ 0 \neq \sum_j i_j \equiv 0 \mod (q-1)}} t_1^{1 - i_1}\cdots t_s^{1-i_s} z(1 - i_1, 1 - i_2,\dots,1 - i_s,t_0).
\end{equation}
\begin{proof}
The equality $z(1,\overset{0}{\dots},1,t_0) = 1$ is obvious.

Now the proof follows in the same elementary way as Goss' original recursive formula for the special polynomials associated to his zeta function; that is,  for $a \in A_+(d)$ one writes $\chi_{i}(a)$ uniquely as $t_i \chi_{i}(b) + \lambda$ with $b \in A_+(d-1)$ and $\lambda \in \FF_q$ and expands as we do now.

Assume $s \geq 1$. Then $z(1,\overset{s}{\dots},1,t_0)$ equals
\begin{eqnarray*}
&&1+\sum_{d \geq 1} t_0^d \sum_{b \in A_+(d-1)} \sum_{\lambda \in \FF_q} \prod_{i = 1}^s \sum_{j_i = 0}^1 (t_i \chi_{i}(b))^{1-j_i} \lambda^{j_i} = \\
&&= 1+\sum_{d \geq 1} t_0^d \sum_{b \in A_+(d-1)} \prod_{i = 1}^s t_i \chi_{i}(b) + \\
&& + \sum_{d \geq 1} t_0^d \sum_{b \in A_+(d-1)} \sum_{\lambda \in \FF_q^\times} \prod_{i = 1}^s \sum_{j_i = 0}^1 (t_i \chi_{i}(b))^{1-j_i} \lambda^{j_i} \\
&(*)&= 1+\sum_{d \geq 1} t_0^d \sum_{b \in A_+(d-1)} \prod_{i = 1}^s t_i \chi_{i}(b) + \\
&(**)& + \sum_{(j_i) \in \{0,1\}^s} \left(\prod_{i = 1}^s t_i^{1-j_i} \right)\sum_{d \geq 1} t_0^d\sum_{b \in A_+(d-1)} \left(\prod_{i = 1}^s \chi_{i}(b)^{1-j_i}\right) \sum_{\lambda \in \FF_q^\times} \lambda^{\sum_{i = 1}^s j_i}.
\end{eqnarray*}

Using the fact that $\sum \lambda^{\sum j_i} = -1$ if $\sum j_i \equiv 0 \mod (q-1)$ and equals zero otherwise, and taking account of the cancellation of $\sum_{d \geq 1} t_0^d \sum_{b \in A_+(d-1)} \prod_{i = 1}^s t_i \chi_{i}(b)$ in $(*)$ and the term arising from the $s$-tuple of all zeros in $(**)$ we deduce the claim.
\end{proof}
\end{proposition}

\begin{corollary}\label{spcor1}
For all non-negative integers $s$, the degree in $t_0$ of $z(1,\overset{s}{\dots},1,t_0)$ is at most $\lfloor s / (q-1) \rfloor$.
\begin{proof}
The proof proceeds by strong induction on $s$. By the recursive formula we see immediately that the claim is true for $0 \leq s < q-1$. Now let $s \geq q-1$, and suppose the claim holds for all $0 \leq j < s$. By the recursive formula and our inductive hypothesis, the highest power of $t_0$ appearing in $z(1,\overset{s}{\dots},1,t_0)$ comes from any $z(1-i_1,1-i_2,\dots,1-i_s,t_0)$ such that exactly $q-1$ of the $i_l$ are equal to $1$, and again by our inductive hypothesis the degree of this latter special polynomial is then $\left\lfloor \frac{s - (q-1)}{q-1} \right\rfloor = \lfloor s/(q-1) \rfloor - 1$. The extra $t_0$ appearing in the recursive formula shows that the degree of $z(1,\overset{s}{\dots},1,t_0)$ is at most $\lfloor s/(q-1) \rfloor$.
\end{proof}
\end{corollary}

\begin{remark}
We will see in Theorem \ref{thm:exdegsp} below that the degree in $t_0$ of $z(1,\overset{s}{\dots},1,t_0)$ actually \textit{equals} $\lfloor s / (q-1) \rfloor$.
\end{remark}

\section{Exact degree}\label{exdegsec}
The following theorem is the main result of this note.

\begin{theorem} \label{thm:exdegsp}
For any $s$ positive integers $\beta_1,\dots,\beta_s$, the exact degree in $t_0$ of the special polynomial
\[z(\beta_1,\dots,\beta_s,t_0) := \sum_{d \geq 0} t_0^d \sum_{a \in A_+(d)} \chi_{1}(a)^{\beta_1}\cdots\chi_{s}(a)^{\beta_s}\]
equals
\begin{equation} \label{degreeeqn}
\phi(\beta_1,\cdots,\beta_s) := \min_{i \geq 0} \left \lfloor \frac{l(p^i \beta_1) + \cdots + l(p^i \beta_s)}{q-1}  \right \rfloor .
\end{equation}
\end{theorem}

The proof may be found in Section \ref{pfexdegsp} below. First, we compare our result with the classical theory.

\subsection{Similarities with Classical $L$-series}
The non-trivial Dirichlet characters on $A$ may be obtained in the following fashion: they are the evaluation maps $\chi$ from $A$ into the algebraic closure of $\FF_q$ defined for all $a \in A$ by
\[\chi(a) := a(\lambda_1)^{\beta_1}\cdots a(\lambda_s)^{\beta_s},\]
where the $\lambda_i$ are roots of distinct monic irreducible polynomials $P_i \in A$ and the $\beta_i$ are positive integers chosen in an appropriate range. The details are easily worked out by the reader. Thus specializing $z(\beta_1,\dots,\beta_{s+1},t_0)$ so that $t_{s+1} = \theta$ and the remaining $t_i = \lambda_i$, for $i = 1,\dots, s$ as above, we obtain the $\beta_{s+1}$-th special polynomial for the Goss-Dirichlet $L$-function associated to the character $\chi$.

The growth rates for sum of digit function $l(\cdot)$ and the base $q$ logarithm $\log_q(\cdot)$ may be easily related - for example, through the simple inequality \begin{equation}\label{digittolog} l(r) \leq (\log_q(r)+1)(q-1),\end{equation} which holds for all positive integers $r$ - and the formula \eqref{degreeeqn} stated above for the degree in $t_0$ of $z(\beta_1,\dots,\beta_s,t_0)$ is strikingly similar to the classical formula for $N(T,\chi)/T$, which we now quickly recall.

For a primitive, complex-valued character $\chi$ of modulus $m$, one sets $N(T,\chi)$ to be the number of zeros of $L(\sigma + it,\chi)$ in the critical strip $0<\sigma<1$ such that $|t| < T$. It is well known that the ``average number of zeros'' $N(T,\chi)/T$ grows like a constant multiple of $\log(T) + \log(m)$. This is quite reminiscent of what follows for the multivariate special polynomials in Theorem \ref{thm:exdegsp} after replacing the sum of digits functions with the upper bound given by $\log_q$ in \eqref{digittolog}.

It should be noted that after specialization of $t_1,\dots,t_s$ at roots of unity (or any algebraic elements), the degree of our special polynomials may decrease. This topic needs further exploration.

\subsection{The proof of Theorem \ref{thm:exdegsp}} \label{pfexdegsp}
In this section and the next we shall use the following obvious lemma without comment.

\begin{lemma}
If $j,k$ are two positive integers such that there is no base $q$ carry-over in the sum $j+k$, then $l(j+k) = l(j)+l(k)$. \hfill $\qed$
\end{lemma}

\begin{proof}[Proof of Theorem \ref{thm:exdegsp}]
As a first step, we demonstrate how to obtain the upper-bound from Cor. \ref{spcor1}. Let $\beta = \sum_{i= 0}^r b_i q^i$ be a positive integer written in base $q$ with $b_r \neq 0$. Observe that we may obtain $z(\beta,t_0)$ from $z(1,\overset{l(\beta)}{\dots},1,t_0)$ by making the following substitutions:
\begin{eqnarray*}
t_1, t_2, \dots, t_{b_0} \mapsto t_1, \\
t_{b_0+1}, t_{b_0 + 2}, \dots, t_{b_0+b_1} \mapsto t_1^q, \\
\vdots \ \\
t_{b_0+b_1+\cdots+ b_{r-1}+1}, t_{b_0+b_1+\cdots+ b_{r-1}+2}, \dots, t_{b_0+b_1+\cdots+ b_{r-1}+b_r} \mapsto t_1^{q^r}.
\end{eqnarray*}
For positive integers $\beta_1,\dots,\beta_s$, one obtains $z(\beta_1,\dots,\beta_s,t_0)$ from $z(1,\overset{l(\beta_1)+l(\beta_2)+\cdots+l(\beta_s)}{\dots},1,t_0)$ similarly.


Now because of the relation $z(\beta_1,\dots,\beta_s,t_0)^{p^i} = z(p^i\beta_1,\dots,p^i\beta_s,t_0^{p^i})$, which holds for all integers $i \geq 0$, we see that the degrees in $t_0$ of $z(\beta_1,\dots,\beta_s,t_0)$ and $z(p^i\beta_1,\dots,p^i\beta_s,t_0)$ are equal. Using the fact that $z(p^i\beta_1,\dots,p^i\beta_s,t_0)$ may be obtained from $z(1,\overset{l(p^i \beta_1) + \dots + l(p^i \beta_s)}{\dots},1,t_0)$, as in the last paragraph, the upper-bound follows.

It remains to argue that the degree cannot be lower, and this is done by a careful choice of specialization to the one variable case; specialization can only decrease the degree in $t_0$.
Let $e = \log_p(q)$. Let $m_0 = 0$, and recursively choose a sequence of positive integers $m_1,\dots, m_{s-1}$ such that $q^{m_j} > q^{m_{j-1}} p^{e-1} \beta_j$. Then for $i = 0,\dots,e-1$ there is no base $q$ carry-over in the sum $p^i \beta_1 + p^i q^{m_1} \beta_2 + \cdots + p^i q^{m_{s-1}} \beta_{s}.$
Hence for these $i$,
\[ l(p^i \beta_1) + l(p^i \beta_2) + \cdots + l(p^i \beta_{s}) = l(p^i (\beta_1 + q^{m_1} \beta_2 + \cdots + q^{m_{s-1}} \beta_{s})).\]
Thus
\[\phi(\beta_1,\dots,\beta_s) = \min_{i \geq 0} \left\lfloor \frac{l(p^i (\beta_1 + q^{m_1} \beta_2 + \cdots + q^{m_{s-1}} \beta_{s}))}{q-1}  \right\rfloor.\]
Now, by B\"ockle \cite[Theorem 1.2 (a)]{B2013}, the exact degree in $t_0$ of the non-zero polynomial
\[ z(\beta_1 + q^{m_1} \beta_2 + \cdots + q^{m_{s-1}} \beta_{s},t_0) = \left.z(\beta_1,\dots,\beta_s,t_0)\right|_{t_1 = t_1, t_2 = t_1^{q^{m_1}}, \dots, t_s = t_1^{q^{m_{s-1}}}}\]
equals $\phi(\beta_1,\dots,\beta_s)$, giving the lower bound.
\end{proof}

\begin{remark}[Logarithmic growth]
The result stated above gives precise meaning to the statement \textit{the degree in $t_0$ of the special polynomials $z(\beta_1,\dots,\beta_s,t_0)$ grows logarithmically in the $\beta_i$.} An important special case, relating to special values of Goss-Dirichlet $L$-functions, is given when one fixes the parameters $\beta_1,\dots,\beta_{s-1}$ and looks at the growth of degrees as $\beta_s$ tends to infinity. Logarithmic growth of this type for the special polynomials associated to $L$-series of $\tau$-sheaves was proved in great generality by B\"ockle, see \cite{Bo02}. Goss used this growth to give analytic continuation to these series using non-Archimedean measure theory. We do not pursue this direction here, but direct the interested reader to \cite{Gosstau}.
\end{remark}

\subsection{The action of Goss' group}\label{gossgp}

Recall from the introduction the definition of Goss' group $S_{(q)}$. We wish to show now that the degrees in $t_0$ of the special polynomials $z(\beta_1,\dots,\beta_s,t_0)$ are invariant under the natural action of $\rho^* := (\rho_1^*,\dots,\rho_s^*) \in (S_{(q)})^s$ given by
\begin{equation}\label{sqaction}
\rho^*(z(\beta_1,\dots,\beta_s,t_0)) := z(\rho_1^*\beta_1,\dots, \rho_s^*\beta_s,t_0).
\end{equation}

\begin{lemma}
Let $\rho^* \in S_{(q)}$ and $k$ a non-negative integer. Then for all $i \geq 0$,
\[l(p^i k) = l(p^i\rho^*(k)).\]
\begin{proof}
Let $e = \log_p(q)$. It suffices to prove this for $i = 0,1,\dots,e-1$, so let $i$ be one such number in this range. Write $k = \sum k_l q^l$ in base $q$.

Clearly for each $l \geq 0$ we may write $p^i k_l = a_l + b_l q$ with $p^i \leq a_l < q$ or $a_l = 0$, and $0 \leq b_l < p^i$. Hence for all $j,k \geq 0$ there is no base $q$ carry over in the sums $a_j+b_k$. Thus $p^i k = \sum_{l \geq 0}a_l q^l + \sum_{l \geq 0} b_l q^{l+1}$ is the sum of two positive integers written in base $q$, and there is no carry over in this decomposition. Hence $l(p^i k) = \sum_{l \geq 0} (a_l+b_l)$. Similarly, $p^i\rho^*(k) = \sum_{l \geq 0} a_l q^{\rho(l)} + \sum_{l \geq 0}b_l q^{\rho(l)+1}$ is the sum of two positive integers written in base $q$, and there is no carry-over in this decomposition. Hence $l(p^i\rho^*(k)) = \sum_{l \geq 0}(a_l + b_l) = l_q(p^i k)$.
\end{proof}
\end{lemma}

\begin{theorem}\label{invgen}
The degree in $t_0$ of $z(\beta_1,\dots,\beta_s,t_0)$ is invariant under the action of $(S_{(q)})^s$ given in \eqref{sqaction}.
\begin{proof}
This is an immediate consequence of the previous result and Theorem \ref{thm:exdegsp}.
\end{proof}
\end{theorem}


\section{Trivial zeros for multivariate special polynomials}
As Theorem \ref{Gossthm} above implies, the special polynomials $z(\beta,t_0)$ have simple ``trivial zeros'' (i.e. simple zeros at $t_0=1$) for all positive integers $\beta$ such that $\beta \equiv 0 \mod (q-1)$, and no zero at $t_0=1$ when $\beta \not\equiv 0 \mod (q-1)$.
Theorem \ref{thmvanchar}, just below, generalizes this, following easily from the next lemma and the specializations in the proof of Theorem \ref{thm:exdegsp} above. See Section \ref{APTZsec} below for an alternate proof of the next lemma using results of Angl\`es-Pellarin in \cite{AngPel1}.



\begin{lemma} \label{spzero}
Suppose $s$ is a positive integer divisible by $q-1$. Then $z(1,\overset{s}{\dots},1,t_0)$ vanishes at $t_0 = 1$.
\begin{proof}
The proof proceeds by induction on $s$ divisible by $q-1$ using the recursive formula from \eqref{sprecursion}.

First suppose $s = q-1$. Then as the recursive formula above shows we have
\[z(1,\overset{s}{\dots},1,t_0) = 1 - t_0.\]
Thus we have a simple zero at $t_0 = 1$.

Now suppose $s > q-1$ is divisible by $q-1$. As the recursive formula shows, apart from $z(1,\overset{0}{\dots},1,t_0) = 1$, only special polynomials occur whose number of indeterminates is positive, divisible by $q-1$ and less than $s$. By our inductive hypothesis, these all vanish at $t_0 = 1$. Collecting these together and calling them $g(t_0)$ we see that
\[z(1,\overset{s}{\dots},1,t_0) = 1 - t_0z(1,\overset{0}{\dots},1,t_0) - g(t_0),\]
with $g(t_0) \in A[t_1,\dots,t_s][t_0]$ vanishing at $t_0 = 1$.
\end{proof}
\end{lemma}

Now we may easily prove the main result of this section. The zeros in the next theorem will be referred to as the \textit{trivial zeros} for our multivariate special polynomials.

\begin{theorem}\label{thmvanchar}
Suppose $\beta_1,\dots,\beta_s$ are positive integers such that $\beta_1+\cdots+\beta_s \equiv 0 \mod (q-1)$, then the special polynomial $z(\beta_1,\dots,\beta_s,t_0)$ has a simple zero at $t_0 = 1$. If $\beta_1,\dots,\beta_s$ are positive integers such that $\beta_1+\cdots+\beta_s \not\equiv 0 \mod (q-1)$, then $z(\beta_1,\dots,\beta_s,t_0)$ does not vanish at $t_0=1$.
\begin{proof}
Let $\beta_1,\dots,\beta_s$ be positive integers such that $\beta_1+\cdots+\beta_s \equiv 0 \mod (q-1)$. As we have argued in the proof of Theorem \ref{thm:exdegsp}, the polynomial $z(1,\overset{l(\beta_1)+\cdots+l(\beta_s)}{\dots},1,t_0)$ may be specialized in the variables $t_i$ to the polynomial $z(\beta_1,\dots,\beta_s,t_0)$. Thus as
\[\beta_1+\cdots+\beta_s \equiv l(\beta_1)+\cdots + l(\beta_s) \mod (q-1)\]
 and evaluation is a ring homomorphism, we obtain a zero at $t_0 = 1$ for the latter polynomial from the former through Lemma \ref{spzero} via specialization.

To see that this zero is simple, we argue by contradiction. Suppose for $\beta_1,\dots,\beta_s$, as above that the zero of $z(\beta_1,\dots,\beta_s,t_0)$ at $t_0 = 1$ is not simple. Specializing $z(\beta_1,\dots,\beta_s,t_0)$ at $t_i = t_1$ for $i = 1,\dots,s$, we obtain a non-simple zero of $z(\beta_1+\cdots+\beta_s,t_0)$, contradicting Theorem \ref{Gossthm}.

Now let $\beta_1,\dots,\beta_s$ be such that $\beta_1+\cdots+\beta_s \not\equiv 0 \mod (q-1)$. To see that we do not have a zero of $z(\beta_1,\dots,\beta_s,t_0)$ at $t_0 = 1$, we argue again by contradiction. Assume $z(\beta_1,\dots,\beta_s,t_0)$ vanishes at $t_0 = 1$. Specializing at $t_i = t_1$ for $i = 1,\dots,s$ we obtain a zero of $z(\beta_1+\cdots+\beta_s,t_0)$, again contradicting Theorem \ref{Gossthm}.
\end{proof}
\end{theorem}

\subsection{Existence of zeros for special polynomials of Goss-Dirichlet $L$-functions}
Theorem \ref{thmvanchar} implies the existence of zeros at $t_0=1$ for special polynomials associated to Goss-Dirichlet $L$-functions (and, similarly, for the negative special values of $L$-series associated to Drinfeld modules over Tate algebras) in the following manner. Let $\beta_1,\dots,\beta_s$ be positive integers, and let $\lambda_1,\dots,\lambda_s \in \overline\FF_q$. We obtain a Dirichlet character $\chi : A \rightarrow \overline\FF_q$ through evaluation via the definition
\[\chi(a) := a(\lambda_1)^{\beta_1}\cdots a(\lambda_s)^{\beta_s},\]
for all $a \in A$.
Thus for non-negative integers $\beta$,
\[z(\chi,\beta,t_0) := \left.\sum_{d \geq 0} t_0^d \sum_{a \in A_+(d)} \chi_1(a)^{\beta_1}\cdots\chi_s(a)^{\beta_s}a^\beta\right|_{\{t_i = \lambda_i\}_{i = 1}^s}\]
is the $\beta$-th special polynomial associated to the $L$-function for $\chi$, and Theorem \ref{thmvanchar} implies that this polynomial has a zero at $t_0 = 1$ whenever $\beta_1+\cdots+\beta_s+\beta \equiv 0 \mod (q-1)$.

Questions of how the degree in $t_0$ of $z(\beta_1,\dots,\beta_s,\beta,t_0)|_{t_{s+1} = \theta}$ can be affected and how non-simple zeros at $t_0 = 1$ can be introduced upon specializing $t_1,\dots,t_s \in \overline\FF_q$ (or in the algebraic closure of $\FF_q(t_1,\dots,t_s)$ as in \cite{APTR14}) are interesting and need further exploration.

\subsection{Angl\`es-Pellarin implies trivial zeros}\label{APTZsec}
Now we describe how Theorem 4 in \cite{AngPel1} gives rise to the trivial zeros of our multivariate special polynomials and implies Proposition \ref{spzero} above. Thus, as was mentioned, the poles of the ``gamma factor'' $\omega$ give trivial zeros of $L$-series in positive characteristic in analogy with classical complex valued $L$-series.

\begin{proof}[Alternate proof of Theorem \ref{spzero}]
Observe that $z(1,\overset{s}{\dots},1,t_0)|_{t_0=1}$ may be obtained from
\[L(\chi_1\cdots\chi_{s+1},1) := \sum_{a \in A_+}\frac{\chi_1(a)\cdots\chi_s(a)\chi_{s+1}(a)}{a}\]
through specialization of the last variable $t_{s+1}$ at $\theta$, as follows from the entireness in the $t_i$ of the $L$-series, defined just above, that was proved in \cite[Prop. 32]{AngPel1}. From Theorem 4 of \cite{AngPel1} we know that taking $s \equiv 0 \mod (q-1)$ we have
\[ \pitilde^{-1} L(\chi_1\cdots\chi_{s+1},1) \omega(t_1)\cdots\omega(t_{s+1}) \prod_{i = 1}^{s+1}\left( 1 - \frac{t_i}{\theta}  \right) \in K[t_1,\dots,t_{s+1}],\]
where the $\omega(t)$ is the Anderson-Thakur function, unique up to a choice of $(q-1)$-th root of $-\theta$ that we fix now, and $\pitilde$ is the fundamental period of the Carlitz module, defined by the equality $\lim_{t \rightarrow \theta}(\theta - t)\omega(t) = \pitilde$. (Note that taking $\delta = 1$ in \cite[Theorem 4]{AngPel1} works in this case.)  In the limit $t_{s+1} \rightarrow \theta$ we learn that
\[\left(z(1, \overset{s}{\dots}, 1,t_0)|_{t_0=1} \right) \omega(t_1) \cdots \omega(t_{s}) \prod_{i = 1}^s\left( 1 - \frac{t_i}{\theta}  \right) \in K[t_1,\dots,t_{s}].\]
But $z(1,\overset{s}{\dots},1,t_0)|_{t_0=1} \in \FF_q[t_1,\dots,t_s]$, and the line above implies that this polynomial must have infinitely many zeros in just the $t_1$ coordinate, for example. Hence it must be identically zero.
\end{proof}

\noindent \textit{Acknowledgments.} Our thanks go to David Goss whose support and interest fueled this work and to both David and Federico Pellarin whose comments have increased the value of this note. We also thank the referee for several beneficial remarks.

\end{document}